\documentclass[11pt]{article}
\usepackage{amsfonts}
\usepackage{amsmath}
\usepackage{amsthm}
\usepackage{amssymb}

\numberwithin{equation}{section}
\newtheorem{theorem}{Theorem}[section]
\newtheorem{lemma}{Lemma}[section]

\title{\bf  Numerical analysis of distributed optimal control problems governed by elliptic variational
inequalities}

\author{\
Mariela Olgu\'in \thanks{Departamento de
Matem\'atica, EFB-FCEIA, Univ. Nacional de Rosario, Avda. Pellegrini 250, S2000BPT Rosario, Argentina.
  E-mail: mcolguin@fceia.unr.edu.ar} \and and \and
Domingo A. Tarzia \thanks{Corresponding author: Departamento de
Matem\'atica-CONICET, FCE, Univ. Austral, Paraguay 1950, S2000FZF
Rosario, Argentina. Tel: +54-341-5223093, Fax:  +54-341-5223001.
E-mail: DTarzia@austral.edu.ar} }

\date{}
\begin{document}
\maketitle

%
\begin{abstract}

A continuous optimal control problem governed by an elliptic variational inequality
was considered in Boukrouche-Tarzia, Comput. Optim. Appl., 53 (2012), 375-392
where the control variable is the internal energy $g$. It was proved the existence and
uniqueness of the optimal control and its associated state system. The objective of
this work is to make the numerical analysis of the above optimal control problem,
through the finite element  method with Lagrange's triangles of type 1. We discretize
the elliptic variational inequality which define the state system and the corresponding cost
functional, and we prove that there exists a discrete optimal control and
its associated discrete state system for each positive $h$ (the parameter of the finite
element method approximation). Finally, we show that the discrete optimal control
and its associated state system converge to the continuous optimal control and its
associated state system when the parameter $h$ goes to zero.

\end{abstract}

\maketitle

\section{Introduction}\label{intro}

We consider a bounded domain $\Omega \subset \mathbb{R}^n $ whose regular boundary $ \partial\Omega= \Gamma_1 \bigcup \Gamma_2$
consists of the union of two disjoint portions $\Gamma_1$ and $\Gamma_2$ with meas ($\Gamma_1$ )$ > 0$ . We consider the
following free boundary problem $(S)$:

\begin{equation}  \label{ec 1} 
   u\geq 0; \hspace{.3cm} u(-\Delta u-g)=0;  \hspace{.5cm} -\Delta u-g \geq 0 \hspace{.3cm} \text{in} \hspace{.5cm}  \Omega;
\end{equation}

\begin{equation}  \label{ec 2} 
   u=b \hspace{.5cm} \text{on}\hspace{.3cm}  \Gamma_1; \hspace{.3cm} -\frac{\partial u}{\partial n} =q \hspace{.3cm} \text{on} \hspace{.5cm}  \Gamma_2;
\end{equation}

\noindent where the function $g$ in (\ref{ec 1}) can be considered as the internal energy in $\Omega$, $b$
is the constant temperature on $\Gamma_1$ and $q$ is the heat flux on $\Gamma_2$. The variational formulation of
the above problem is given as: Find  $u= u_g \in  K$ such that

\begin{equation} \label{ec 3}  
   a(u,v-u)\geq (g, v-u)_H - \int_{\Gamma_2} q(v-u) \,\, ds,  \hspace{1cm} \forall v \in K,
\end{equation}

\noindent where
\begin{equation*}
V=H^1(\Omega), \hspace{.5cm} K=\{v \in V: \, v\geq 0 \,\,\text{in $\Omega$}, \,v{/ \Gamma_1} = b\}, \hspace{.5cm} V_0=\{ v \, \in V : v{/ \Gamma_1} =0\},
\end{equation*}

\begin{equation*}
 H= L^2(\Omega), \hspace{0.5cm} Q=L^2(\Gamma_2), \hspace{0.5cm} (u,v)_Q = \int_{\Gamma_2} u\,v \,\, ds \hspace{0.5cm} \forall \,\, u,v \in Q,
\end{equation*}

\begin{equation*}
 a(u,v)= \int_\Omega \nabla u. \nabla v \,\, dx \hspace{.5cm} \forall\,\, u,v \in V,  \hspace{0.5cm} (u,v)_H= \int_\Omega u\, v \,\, dx \hspace{0.5cm} \forall \,\, u,v \in H.
\end{equation*}

\noindent We note that $a$ is a bilinear, continuous, symmetric on $V$ and a coercive form on $V_0$ \cite{Kinderhlerer80}
, that is to say: there exists a constant $\lambda > 0$ such that

 \begin{equation} \label{ec 4}  
\,\,a(v,v) \geq \lambda \,\lVert v \rVert^2_V  \hspace{0.5cm} \forall \,\,v \in V_0.
\end{equation}

\noindent In \cite{Boukrouche12}, the following continuous distributed optimal control problem associated
with $(S)$ or the elliptic variational inequality $(1.3)$ was considered:

\vspace{.2cm}

\noindent Problem $(P)$: Find the continuous distributed optimal control $g_{op} \, \in H$ such that

\begin{equation}  \label{ec 5}  
  J(g_{op})=\min_{g \in H} J(g)
\end{equation}

\noindent where the quadratic cost functional $J:H \rightarrow \mathbb{R}^+_0$ is defined by:

\begin{equation} \label{ec 6}  
   J(g)= \frac{1}{2} \lVert u_g \rVert^2_H + \frac{M}{2} \lVert g \rVert^2_H
\end{equation}

\noindent with $M > 0$ a given constant and $u_g$ is the corresponding solution of the elliptic
variational inequality (\ref{ec 3}) associated to the control $g$.

Several continuous optimal control problems are governed by elliptic variational inequalities, for example:
the process of biological waste-water treatment; reorientation of a satellite by propellers;
and economics: the problem of consumer regulation of a monopoly, etc. There exist an abundant
 literature for optimal control problems \cite{Barbu84,Lions68,Troltzsch10}, for optimal control problems
governed by elliptic variational equalities or inequalities
\cite{Adams98,Ait08,BenBelgacem03,Bergounioux97,Bergounioux97a,Bergounioux97b,Bergounioux00,Boukrouche12,De Los Reyes11,
DeLosReyes14,Gariboldi03,Haller09,Hintermuller01,Hintermuller09,Hintermuller13,Ito00,Kunisch12,Mignot76,Mignot84,
Ye04a,Ye04,Ye09},
for numerical analysis of variational inequalities or optimal control problems \cite{Beuchler12,Burman14, Casas02,Casas05,Casas06,Casas08,Deckelnick09,Deckelnick07,Djoko08, Falk74, Gamallo11,Glowinski84, Hintermuller09a,Hinze05,Hinze09,Hinze09a,Mermri12,Tarzia96,Tarzia99,Tarzia14,Yan12}, and
for the numerical analysis of optimal control problems governed by an elliptic variational inequality there exist a few numbers of papers
\cite{Abergel88,Haslinger86,Hintermuller08,Meyer13}.

The objective of this work is to make the numerical analysis of the optimal control problem $(P)$ which
is governed by the elliptic variational inequality (\ref{ec 3}) by proving the convergence of a discrete
solution to the continuous optimal control problems.

In Section 2, we establish the discrete elliptic variational inequality (\ref{ec 9}) which is the
discrete formulation of the continuous elliptic variational inequality (\ref{ec 3}), and we obtain
that these discrete problems have unique solutions for all positive $h$. Moreover, on the
adequate functional spaces these solutions are convergent when $h\rightarrow 0^+$ to the solutions of
the continuous elliptic variational inequality (\ref{ec 3}).

In Section 3, we define the discrete optimal control problem (\ref{ec 26}) corresponding to
continuous optimal control problem (\ref{ec 5}).
We prove the existence of a discrete solution for the optimal control problem ($P_h$) for each parameter
$h$ and we obtain the convergence of this family with its corresponding discrete state system to the continuous optimal control with the
corresponding continuous state system of the problem ($P$).

\vspace{0.5cm}
\section{Discretization of the problem (S)}

Let $\Omega\subset \mathbb{R}^n$ a bounded polygonal domain; $b$ a positive constant and $\tau_h$ a regular triangulation
with Lagrange triangles of type 1, constituted by affine-equivalent finite elements of class $C^0$ over $\Omega$ being
$h$ the parameter of the finite element approximation which goes to zero \cite{Brenner94, Ciarlet02}. We
take $h$ equal to the longest side of the triangles $T \in \tau_h$ and we can approximate the sets $V$ and $K$ by:

$$V_h= \{v_h \in C^0(\overline{\Omega}): v_h /_T \in \mathbb{P}_1(T), \,\,\forall \,T \,\in \tau_h\}$$

$$V_{h0}= \{v_h \in C^0(\overline{\Omega}): {v_h} /_{\Gamma_1} = 0;\, v_h /_T \in \mathbb{P}_1(T), \,\,\forall \,T \,\in \tau_h\}$$

\noindent and

$$K_h= \{ v_h \in C^0(\overline{\Omega}) : v_h\geq 0,\,\, {v_h}_{/_ {\Gamma_1}}=b,\,\, {v_h}/_ T \in \mathbb{P}_1(T) \,\,\forall \hspace{.2cm} T \, \in \tau_h  \}$$

\noindent where $\mathbb{P}_1(T)$ is the set of the polynomials of degree less than or equal to $1$ in the triangle $T$. Let $\Pi_h : V\rightarrow V_h$
be the corresponding linear interpolation operator and  $c_0 > 0$  a constant (independent of the parameter $h$) such that
\cite{Brenner94}:

\begin{equation} \label{ec 7}  
  \rVert v-\Pi_h(v)\lVert_H \, \leq \, c_0 \, h^r \,\rVert v \lVert_r \hspace{.4cm} \forall \, v \in H^r(\Omega), \,\,\,1< r \leq 2
\end{equation}

\begin{equation} \label{ec 8}  
  \rVert v-\Pi_h(v)\lVert_V \, \leq \, c_0 \, h^{r-1} \,\rVert v \lVert_r \hspace{.4cm} \forall \, v \in H^r(\Omega), \,\,\,1< r \leq 2.
\end{equation}

\noindent The discrete variational inequality formulation $(S_h)$ of the system $(S)$ is defined as: Find $u_{hg} \in K_h$ such that

\begin{equation} \label{ec 9}  
    a(u_{hg},v_h-u_{hg})\geq (g, v_h-u_{hg})_H - \int_{\Gamma_2} q(v_h-u_{hg}) d\gamma,  \hspace{1cm} \forall v_h \in K_h.
\end{equation}

\begin{theorem} 

Let $g\in H$, $b > 0$ and $q \in Q $ be, then there exist unique solution of the
problem $(S_h)$ given by the elliptic variational inequality (\ref{ec 9}).
\end{theorem}

\begin{proof}
 It follows from the application of Lax-Milgram Theorem \cite{Kinderhlerer80, Lions67}.
\end{proof}

\begin{lemma}  Let $g_1,\,g_2\,\in H$, and $u_{hg_1}$, $u_{hg_2} \,\in K_h$ be the solutions of $(S_h)$ for
$g_1$ and $g_2$ respectively, then we have that:

\begin{enumerate}
 \item [a)] there exist a constant $C$ independent of $h$ such that:
 \begin{equation}\label{ec 10}  
     \rVert u_{hg}\lVert_V \leq C, \hspace{.5cm} \forall \, h > 0;
  \end{equation}

 \item [b)]
  \begin{equation} \label{ec 11} 
      \rVert u_{hg_2}-u_{hg_1} \lVert_V \, \leq \, \frac{1}{\lambda} \rVert g_2 - g_1 \lVert_H  \,\,\,\forall \,\, h > 0;
  \end{equation}

\item [c)] if $g_n \rightharpoonup g$ in $H$ weak, then $u_{hg_n} \rightarrow u_{hg}$ in $V$ strong for each fixed $h >0$.
\end{enumerate}

\end{lemma}

\begin{proof}

 a) If we consider $v_h=b \in K_h$ in the discrete elliptic variational inequality (\ref{ec 9}) we have:

\begin{equation*}
  \lambda\,\rVert u_{hg}-b\lVert^2_V \leq a(u_{hg},u_{hg}-b) \leq (g,u_{hg}-b)_H + (q,b-u_{hg})_Q
\end{equation*}

\begin{equation*}
 \leq (\rVert g \lVert_H \,+ \,\rVert q \lVert_Q \,\rVert \gamma_0 \lVert) \rVert u_{hg}-b \lVert_V
\end{equation*}

\noindent where $\gamma_0$ is the trace operator and therefore (\ref{ec 10}) holds.

\vspace{.2cm}

\noindent b) As $u_{hg_1}$ and $u_{hg_2}$ are respectively the solutions of discrete elliptic variational inequalities
(\ref{ec 9}) for $g_1$ y $g_2$, we have:

\begin{equation} \label{ec 12} 
   a(u_{hg_i}, v_h - u_{hg_i})\geq (g_i, v_h - u_{hg_i})_H -  (q, v_h-u_{hg_i})_Q,  \hspace{1cm} \forall v_h \in K_h
\end{equation}

 \noindent for $i=1,2$. By coerciveness of $a$ we deduce:

$$\lambda \rVert u_{hg_2}-u_{hg_1} \lVert^2_V \, \leq  a(u_{hg_2}-u_{hg_1},u_{hg_2}-u_{hg_1}) \leq (g_2-g_1,u_{hg_2}-u_{hg_1})_H $$
 $$\leq  \rVert g_2 - g_1 \lVert_H  \rVert u_{hg_2}-u_{hg_1} \lVert_V \hspace{.4cm} \forall \,\,\,h>0,$$

\noindent thus (\ref{ec 11}) holds.

\vspace{.1cm}

\noindent c) Let $h >0$ be. From item a) we have that $\lVert u_{hg_n}\rVert \leq \,C \,\,\,\forall \,n$, then there exist $\eta \in V$ such that
$u_{hg_n} \rightharpoonup \eta$ in $V$ weak (in $H$ strong). If we consider the discrete elliptic inequality (\ref{ec 9}) we have:

\begin{equation*}
 a(u_{hg_n}, v_h -u_{hg_n}) \geq (g_n, v_h -u_{hg_n})_H- (q, v_h -u_{hg_n})_Q
\end{equation*}

\noindent and using that $a$ is a lower weak semi-continuous application then, when $n$ goes to infinity, we obtain that:

\begin{equation*}
 a(\eta, v_h -\eta) \geq (g, v_h -\eta)_H- (q, v_h -\eta)_Q
\end{equation*}

\noindent and from uniqueness of the solution of problem $(S_h)$, we deduce that $\eta=u_{hg} \in K_h$.

\noindent Now, it is easily to see that:

\begin{equation*}
 a(u_{hg_n}-u_{hg}, u_{hg_n}-u_{hg}) \leq -(g-g_n,u_{hg_n}-u_{hg})_H
\end{equation*}

\noindent and from the coerciveness of $a$ we obtain

\begin{equation*}
 \lambda \lVert u_{hg_n}-u_{hg} \rVert^2_V \leq (g-g_n,u_{hg_n}-u_{hg})_H.
\end{equation*}

\noindent As $u_{hg_n}\rightarrow u_{hg}$ in $H$ and $g_n \rightharpoonup g$ in $H$, by pass to the limit when $n\rightarrow \infty$
in the previous inequality, we obtain

 \begin{equation*}
                       \lim_ {n \rightarrow \infty} \rVert u_{hg_n}-u_g \lVert_V =0.
 \end{equation*}

\end{proof}

Henceforth we will consider the following definitions \cite{Boukrouche12}: Given $\mu \in [0,1]$ and $g_1, g_2 \in H$, we have the convex
combinations of two data

\begin{equation} \label{ec 13} 
  g_3(\mu) = \mu \, g_1 + (1-\mu) g_2 \hspace{.2cm} \in H,
\end{equation}

\noindent the convex combination of two discrete solutions

\begin{equation}  \label{ec 14} 
 u_{h 3}(\mu)= \mu \, u_{hg_1} + (1- \mu) u_{hg_2} \hspace{.2cm}\in K_h
\end{equation}

\noindent and we define $u_{h 4}(\mu)$ as the associated state system which is the solution of the discrete elliptic variational
inequality (\ref{ec 9}) for the control $g_3(\mu)$.

\vspace{.5cm}

\noindent Then, we have the following properties:

\begin{lemma} Given the controls $g_1, g_2 \in H$, we have that:

a)\begin{equation}  \label{ec 15} 
    \rVert u_{h 3}\lVert^2_H = \mu\,\rVert u_{hg_1}\lVert^2_H + (1-\mu)\,\rVert u_{hg_2}\lVert^2_H - \mu\,(1-\mu)\rVert u_{hg_2}-u_{hg1}\lVert^2_H
  \end{equation}

b) \begin{equation} \label{ec 16} 
     \rVert g_3(\mu)\lVert^2_H = \mu\,\rVert g_1\lVert^2_H + (1-\mu)\,\rVert g_2\lVert^2_H - \mu\,(1-\mu)\rVert g_2 - g_1\lVert^2_H
   \end{equation}

\end{lemma}

\begin{proof}

\noindent a) From the definition (\ref{ec 14}) we get

   \begin{equation*}
     \rVert u_{h 3}\lVert^2_H = \mu^2 \,\rVert u_{hg_1}\lVert^2_H + (1-\mu)^2 \,\rVert u_{hg_2}\lVert^2_H +2 \,\mu\,(1-\mu) \,(u_{hg_1}, u_{hg_2})_H
   \end{equation*}

\noindent and
\begin{equation*}
   \rVert u_{hg_2}- u_{hg_1}\lVert^2_H = \rVert u_{hg_2}\lVert^2_H + \rVert u_{hg_1}\lVert^2_H - 2 (u_{hg_1}, u_{hg_2})_H,
\end{equation*}

\noindent then we conclude (\ref{ec 15}).

\vspace{.1cm}

\noindent b) It follows from a similar method to the part a).
\end{proof}

\begin{theorem} 

 If $ u_g$ and $u_{hg}$ be the solutions of the elliptic variational inequalities (\ref{ec 3}) and (\ref{ec 9}) respectively for the control $g \in H$,
 then $u_{hg}$ converge to $u_{g}$ in $V$ strong when $h \rightarrow 0^+$.
\end{theorem}

\begin{proof}

\noindent From Lemma 2.1 we have that there exist a constant $C > 0$ independent of $h$ such that $ \rVert u_{hg} \lVert_V \leq C \hspace{0.5cm} \forall \,\,h>0,$
then we conclude that there exists $\eta \in V$ so that
 $u_{hg} \rightharpoonup \eta$ in $V$ weak as $h \rightarrow 0^+$ and $\eta \in K$.  On the other hand, given $v \in K$
 there exist ${v^*_h}$ such that $v^*_h \in K_h$ for each $h$ and $v^*_h\rightarrow v$ in $V$ strong when
 $h$ goes to zero. Now, by considering $v^*_h \in K_h$ in the discrete elliptic variational inequality (\ref{ec 9}) we get:

 \begin{equation} \label{ec 23} 
   a(u_{hg}, u_{hg}) \leq a(u_{hg},v^*_h) - (g, v^*_h-u_{hg}) + (q,v^*_h-u_{hg})_Q
 \end{equation}

\noindent and when we pass to the limit as $h\rightarrow 0^+$ in (\ref{ec 23}) by using that the bilinear form $a$ is lower
weak semicontinuous in $V$ we obtain:

$$a(\eta, \eta) \leq a(\eta, v) - (g, v-\eta) + (q, v-\eta)_Q$$

\noindent that it is to say:

$$ a(\eta, v - \eta) \geq (g, v - \eta) - (q, v-\eta)_Q \hspace{.4cm} \forall \,v\in K$$

\noindent and, from the uniqueness of the solution of the discrete elliptic variational inequality (\ref{ec 3}), we obtain that $ \eta=u_g.$

\vspace{.1cm}

\noindent Now, we will prove the strong convergence. If we consider $v=u_{hg} \in K_h \subset K$ in the elliptic variational
inequality (\ref{ec 3}) and $v_h =\Pi_h(u_g) \in K_h$ in (\ref{ec 9}), from the coerciveness of $a$ and by some mathematical
computation, we obtain that:

\begin{equation*}
      \lambda \rVert u_{hg} - u_g\lVert^2_V \, \leq a(u_{hg}-u_g,u_{hg}-u_g)
    \end{equation*}

\begin{equation} \label{ec 24}  
 \leq a(u_{hg}, \Pi_h(u_g)-u_g)  - (g, \Pi_h(u_g)-u_g)+ (q, \Pi_h(u_g)-u_g)_Q
\end{equation}

\noindent then by pass to the limit when $h\rightarrow 0^+$ it results that $\lim_{h\rightarrow 0^+} \lVert u_{hg}-u_g\rVert_V =0.$
\end{proof}

\section{Discretization of the optimal control problem}

Now, we consider the continuous optimal control problem which was established in (\ref{ec 5}). The associated
discrete cost functional  $J_h: H \rightarrow \mathbb{R}^+_0\, $is defined by the following expression:

 \begin{equation} \label{ec 25} 
   \,\,  J_h(g)= \frac{1}{2} \lVert u_{hg} \rVert^2_H + \frac{M}{2} \lVert g \rVert^2_H
\end{equation}

\noindent and we establish the discrete optimal control problem $(P_h)$ as: Find $g_{op_h} \,\in H$ such that
\begin{equation} \label{ec 26}  
  J_h(g_{op_h})=\min_{g \,\in\,H} J_h(g)
\end{equation}
where $u_{hg}$ is the associated state system solution of the problem $(S_h)$  which was described
for the discrete elliptic variational inequality (\ref{ec 9}) for a given control $g \in H$.

\begin{theorem}  
  Given the control $g \in H$, we have:

  \vspace{.1cm}
  a) \begin{equation*}
                       \lim_ {\lVert g \rVert_H \rightarrow \infty} J_h(g) = \infty.
     \end{equation*}

  b) $ J_h(g) \geq \frac{M}{2} \rVert g\lVert^2_H - C \,\rVert g\lVert_H$ \, for some constant $C$ independent of $h$.

  \vspace{.1cm}

  c) The functional $J_h$ es a lower weakly semi-continuous application in $H$.

  \vspace{.1cm}


  \vspace{.1cm}

  d) There exists a solution of the discrete optimal control problem (\ref{ec 26}) for all $h > 0$.

\end{theorem}

\begin{proof}
a) From the definition of $J_h(g)$ we obtain a) and b).

\vspace{.1cm}

c) Let $g_n\rightharpoonup g$ in $H$ weak, then by using the equality $\rVert g_n \lVert^2_H = \rVert g_n - g \lVert^2_H - \rVert g \lVert^2_H + 2 (g_n, g)_H$
we obtain that $\rVert g \lVert_H \leq  \liminf_{n\rightarrow \infty} \, \rVert g_n \lVert^2_H$. Therefore, we have

\begin{equation*}
  \liminf_{n\rightarrow \infty} J_h(g_n) \geq \frac{1}{2} \rVert u_{hg}\lVert^2_H + \frac{M}{2} \rVert g \lVert^2_H = J_h(g).
\end{equation*}

\vspace{.1cm}

d) It follows from \cite{Lions68}.
 \end{proof}

 \begin{lemma} 

 If the continuous state system has the regularity $u_g \in H^r(\Omega) \,(1< r \leq 2)$ then we have the
 following estimations $\forall g \in H$:

 \vspace{.1cm}

 a) \begin{equation} \label{ec 27}  
       \rVert u_{h g}-u_g\lVert_V \leq C h^{\frac{r-1}{2}},
 \end{equation}

 \vspace{.1cm}

 b)  \begin{equation} \label{ec 28}  
       \lvert J_h(g)- J(g)\lvert \leq C h^{\frac{r-1}{2}}.
     \end{equation}

 \noindent where $C$'s are constants independents of $h$.

\end{lemma}

 \begin{proof}

 \vspace{.1cm}

 a) As $u_g \in K$, we have that $\Pi_h(u_g) \in K_h \subset K$. If we consider $v_h= \Pi_h(u_g)$ in (\ref{ec 9}),
 by using the inequalities (\ref{ec 24}), we obtain:

 \begin{equation*}
  \lambda \rVert u_{hg} - u_g\lVert^2_V \, \leq  a(u_{hg} - u_g, u_{hg} - u_g)
 \end{equation*}

 \begin{equation*}
  \leq \, a(u_{hg}, \Pi_h(u_g)-u_g) - (g, \Pi_h(u_g)-u_g) + \int_{\Gamma_2} q(\Pi_h(u_g)-u_{g}) \, d\gamma
 \end{equation*}

 \begin{equation*}
  \leq C \rVert \Pi_h(u_g)-u_g \lVert_V \,\leq  C \rVert u_g \lVert_r \, h^{r-1} \, \leq C h^{r-1},
 \end{equation*}

 \noindent and then (\ref{ec 27}) holds.

 \vspace{.1cm}

 b) From the definitions of $J$ and $J_h$, it results:

  \begin{equation*}
   J_h(g)-J(g) = \frac{1}{2}\,(\rVert u_{hg}\lVert^2_H- \rVert u_g\lVert^2_H)= \frac{1}{2} [\rVert u_{hg}-u_g\lVert^2_H \, + \, (u_g, u_{hg} - u_g)]
  \end{equation*}

\noindent and therefore

 \begin{equation*}
   \lvert J_h(g)-J(g) \rvert \leq \, (\frac{1}{2}\rVert u_{hg}-u_g\lVert_H \,+\, \rVert u_g\lVert_H) \, \rVert u_{hg}-u_g\lVert_H \leq C \, h^{\frac{r-1}{2}}.
 \end{equation*}

 \end{proof}

Following the idea given in \cite{Boukrouche12} we define an open problem:
 Given the controls $g_1, g_2 \in H$,
    \begin{equation} \label{ec 29} 
       0 \leq u_{h 4} (\mu) \leq u_{h 3}( \mu) \hspace{0.1cm} in \hspace{0.1cm} \Omega, \hspace{.5cm} \forall \,\, \mu \in [0, 1], \,\forall h>0,
    \end{equation}
\noindent or

\begin{equation} \label{ec 30} 
       \rVert u_{h 4} (\mu) \lVert_H \leq \rVert u_{h 3}( \mu)\lVert_H  \hspace{.5cm} \forall \,\, \mu \in [0, 1], \,\forall h>0.
    \end{equation}

\noindent\textbf{Remark 1}: We have that $(\ref{ec 29})  \Rightarrow  (\ref{ec 30})$.

\noindent\textbf{Remark 2}: The equivalent inequality $(\ref{ec 29})$ for the continuous optimal control problem
$(P)$ is true, that is \cite{Boukrouche12}: for all $g_1, g_2 \in H$,

\begin{equation} \label{ec 31} 
       0 \leq u_{4} (\mu) \leq u_{3}( \mu) \hspace{0.1cm} in \hspace{0.1cm} \Omega, \hspace{.5cm} \forall \,\, \mu \in [0, 1].
    \end{equation}

\noindent where $u_3(\mu)= \mu u_{g_1} + (1- \mu) u_{g_2} \,\, \in \,\,K, \,\, u_{g_i} (i=1,2)$ is the unique solution of the
elliptic variational inequality $(\ref{ec 3})$ when we consider $g_i$ instead of $g$, and $u_4(\mu)$ is the unique solution of the
elliptic variational inequality $(\ref{ec 3})$ when we consider $g_3(\mu)$ instead of $g$.

\noindent\textbf{Remark 3}: If (\ref{ec 30}) (or (\ref{ec 29})) is true, then the functional $J_h$ is H-elliptic and a strictly
convex application because we have

  $$\mu J_h(g_1) + (1-\mu) J_h(g_2)-J_h(g_3(\mu))$$

 $$= \frac{\mu(1-\mu)}{2}\lVert u_{h g_2}-u_{h g_1} \rVert^2_H + \frac{M}{2} \, \mu (1-\mu)\, \lVert g_2-g_1 \rVert^2_H + \frac{1}{2} [\lVert u_{h 3} \rVert^2_H- \lVert u_{h 4} \rVert^2_H]$$
 $$\geq \frac{\mu(1-\mu)}{2}\lVert u_{h g_2}-u_{h g_1} \rVert^2_H + \frac{M}{2} \, \mu (1-\mu)\, \lVert g_2-g_1 \rVert^2_H > 0$$

\noindent and therefore, the uniqueness for the discrete optimal control problem $(P_h)$ in the theorem $3.1$ holds.

\begin{theorem} 

 Let $u_{g_{op}} \in K$ be the continuous state system associated to the optimal control $g_{op} \in H$
 which is the solution of the continuous distributed optimal control problem (\ref{ec 5}). If, for each
 $h>0$, we choose an optimal control $g_{op_h} \in H$ which is the solution of the discrete distributed
 optimal control problem (\ref{ec 26}) and its corresponding discrete state system $u_{h\,g_{op_h}} \in K_h$,
 we obtain that:

$$ u_{h\,g_{op_h}}\rightarrow u_{g_{op}} \hspace{.5cm} \text{on} \hspace{.5cm} V \hspace{.2cm} \text{strong} \hspace{0.4cm} \text{and} \hspace{0.4cm} g_{op_h} \rightarrow g_{op}  \hspace{0.3cm} \text{on} \hspace{.2cm}  H \hspace{.2cm} \text{strong} \hspace{.2cm} \text{when} \hspace{.2cm} h\rightarrow 0^+.$$
\end{theorem}

\begin{proof}

Let be $h >0$ and $g_{op_h}$ a solution of (\ref{ec 26}), and $u_{h\,g_{op_h}} $ its associated
discrete optimal state system which is the solution of the discrete elliptic variational inequality (\ref{ec 9}) for each $h >0$.
From (\ref{ec 25}) we have that for all $g \in H$

 \begin{equation*}
  J_h(g_{op_h})= \frac{1}{2} \lVert u_{h\,g_{op_h}} \rVert^2_H + \frac{M}{2} \lVert g_{op_h} \rVert^2_H\, \leq \, \frac{1}{2} \lVert u_{hg} \rVert^2_H + \frac{M}{2} \lVert g \rVert^2_H.
 \end{equation*}

\noindent Then, if we consider $g=0$  and $u_{h0}$ his corresponding associated state system, it results that:

\begin{equation*}
  J_h(g_{op_h})= \frac{1}{2} \lVert u_{h\,g_{op_h}} \rVert^2_H + \frac{M}{2} \lVert g_{op_h} \rVert^2_H\, \leq \, \frac{1}{2} \lVert u_{h0} \rVert^2_H.
 \end{equation*}

\noindent From the Lemma 2.1 we have that $\lVert u_{h\,0} \rVert_H \leq C \hspace{0.4cm} \forall \hspace{0.2cm} h$, then we can obtain:

  \begin{equation} \label{ec 31}  
        \lVert u_{h\,g_{op_h}} \rVert_H \leq C \hspace{0.4cm} \forall \hspace{0.2cm} h>0
     \end{equation}

\noindent and

     \begin{equation}  \label{ec 32}  
        \lVert g_{op_h} \rVert_H \, \leq \, \frac{1}{M} \, \lVert u_{h\,0} \rVert_H \leq \frac{1}{M}\, C  \hspace{0.4cm} \forall \hspace{0.2cm} h>0.
     \end{equation}

\vspace{.2cm}

\noindent If we consider $v_h=b \in K_h$ in the inequality (\ref{ec 9}) for $g_{op_h}$, we obtain:

\begin{equation} \label{ec 33}  
  a(u_{h\,g_{op_h}}, b-u_{h\,g_{op_h}}) \geq (g_{op_h}, b-u_{h\,g_{op_h}}) - (q, b-u_{h\,g_{op_h}})_Q,
\end{equation}

\noindent therefore:

\begin{equation} \label{ec 34}  
  a(u_{h\,g_{op_h}}-b, u_{h\,g_{op_h}}-b) \leq (g_{op_h}, u_{h\,g_{op_h}}-b) - (q, u_{h\,g_{op_h}}-b)_Q,
\end{equation}

\noindent and from the coerciveness of the application $a$ we have that $\lVert u_{h\,g_{op_h}}-b \rVert_V \leq C$ and in consequence
$\lVert u_{h\,g_{op_h}} \rVert_V \leq C$.

\noindent Now we can say that there exist $\eta \,\, \in \,\,V$ and $f \,\, \in \,\,H$ such that $u_{h\,g_{op_h}}\rightharpoonup \eta$ in $V$ weak
(in $H$ strong), and $g_{op_h}\rightharpoonup f$ in $H$  weak when $h\rightarrow 0^+$. Then, $\eta{/ \Gamma_1} = b$ and  $\eta \geq 0$ in $\Omega$
i.e., $\eta \in K$.

Let given $v \in K$, there exist $v_h \in K_h$ such that $v_h\rightarrow v$ in $V$ strong when $h\rightarrow 0^+$. Then,
if we consider the variational elliptic inequality (\ref{ec 9}) for $g= g_{op_h}$ we have:

\begin{equation} \label{ec 35}  
  a(u_{h\,g_{op_h}}, v_h) \geq a(u_{h\,g_{op_h}}, u_{h\,g_{op_h}}) + (g_{op_h}, v_h - u_{h\,g_{op_h}}) - (q, v_h-u_{h\,g_{op_h}})_Q.
\end{equation}

\noindent Taking into account that the application $a$ is a lower weak semi-continuous application in $V$ and by pass to the limit when
$h$ goes to zero in (\ref{ec 33}) we obtain that:

\begin{equation*}
 a (\eta, v-\eta) \geq (f, v-\eta) - (q, v-\eta)_Q, \hspace{1cm} \forall \,\,v \in K
\end{equation*}

\noindent and by the uniqueness of the solution of the problem given by the elliptic variational inequality
(\ref{ec 3}), we deduce that $\eta = u_f$ .

Finally, the norm on $H$ is a lower semi-continuous application in the weak topology, then we can prove that:

\begin{equation*}
   J(f)=\frac{1}{2} \lVert u_f \rVert^2_H + \frac{M}{2} \lVert f \rVert^2_H
    \leq \liminf_{h\rightarrow 0^+} J_h(g_{op_h}) \leq \liminf_{h\rightarrow 0^+} J_h(g) = \frac{1}{2} \lim_{h\rightarrow 0^+} \lVert u_{hg} \rVert^2_H + \frac{M}{2} \lVert g \rVert^2_H
\end{equation*}

\begin{equation*}
 = \frac{1}{2} \lVert u_g \rVert^2_H + \frac{M}{2} \lVert g \rVert^2_H = J(g), \hspace{1cm} \forall \,g \in H
\end{equation*}

\noindent and because the uniqueness of the optimal problem (\ref{ec 5}), it results that $f=g_{op}$ and $\eta = u_{g_{op}}.$

Now, if we consider $v=u_{h\,g_{op_h}} \in K_h \subset K$ in the elliptic variational inequality (\ref{ec 3}) for the control $g_{op}$ and we define
$z_h= u_{h\,g_{op_h}} - u_{g_{op}}$, we have that:
\begin{equation*}
  a(z_h,z_h) \leq a(u_{h\,g_{op_h}}, u_{h\,g_{op_h}})- a(u_{h\,g_{op_h}}, u_{g_{op}}) - (g_{op}, u_{h\,g_{op_h}} - u_{g_{op}}) + (q,u_{h\,g_{op_h}} - u_{g_{op}} )_Q,
\end{equation*}

\noindent and by consider $v= \Pi_h (u_{g_{op}}) \in K_h$ for $g= g_{op_h}$ in the inequality (\ref{ec 9})  we obtain:

\begin{equation*}
 a(u_{h\,g_{op_h}}, u_{h\,g_{op_h}}) \leq -(g_{op_h},\Pi_h (u_{g_{op}})- u_{h\,g_{op_h}}) + (q,\Pi_h (u_{g_{op}})- u_{h\,g_{op_h}} )_Q + a(u_{h\,g_{op_h}}, \Pi_h (u_{g_{op}})).
\end{equation*}

\noindent and then by the coerciveness of $a$ we get

\begin{equation*}
 \lambda \, \lVert z_h \rVert^2_V \leq (q,\Pi_h (u_{g_{op}})- u_{g_{op}} )_Q + a(u_{h\,g_{op_h}}, \Pi_h (u_{g_{op}})- u_{g_{op}} )
\end{equation*}
\begin{equation} \label{ec 36} 
+(g_{{op}_h}- g_{op}, u_{h\,g_{op_h}}- u_{g_{op}} ) - (g_{op}, \Pi_h (u_{g_{op}})- u_{g_{op}})
\end{equation}

\noindent When we pass to the limit as $h\rightarrow 0⁺$ in (\ref{ec 34}) and by using the strong convergence of $u_{h\,g_{op_h}}$
to $u_{g_{op}}$ on $H$ and the weak convergence of $g_{op_h}$ to $g_{op}$ on $H$ , we have:

\begin{equation} \label{ec 37} 
  \lim_ {h \rightarrow 0^+} \rVert  u_{g_{op}} - u_{h\,g_{op_h}} \lVert_V =0.
\end{equation}

The strong convergence of the optimal controls $g_{op_h}$ to $g_{op}$ is obtained by using Theorem 3.1 and $g_{op_h} \rightharpoonup g_{op}$ weakly on $H$, i.e.

\begin{equation*}
 J(g_{op}) = \frac{1}{2} \lVert u_{g_{op}} \rVert ^2_H + \frac{M}{2} \lVert g_{op} \rVert ^2_H \leq \liminf_{h \rightarrow 0^+} J_h(g_{op_h})
\end{equation*}

 \begin{equation*}
 \leq \liminf_{h\rightarrow 0^+} J_h(g_{op}) = \liminf_{h\rightarrow 0^+} \frac{1}{2} \lVert u_{g_{op}} \rVert ^2_H + \frac{M}{2} \lVert g_{op} \rVert ^2_H = J(g_{op},)
\end{equation*}
\noindent  then $\lim_{h\rightarrow 0} \lVert g_{op_h}\rVert_H = \lVert g_{op}\rVert_H$ and therefore  $\lim_ {h \rightarrow 0^+} \rVert  g_{{op}_h} - g_{op} \lVert_H =0$.

\end{proof}

\section{Conclusions}

We have proved the convergence of a discrete optimal control and its corresponding
discrete state system governed by a discrete elliptic variational inequality to the
continuous optimal control and its corresponding continuous state system which is also governed by
a continuous elliptic variational inequality by using the finite element method
with Lagrange's triangles of type 1.
Moreover, it is an open problem to obtain the error estimates as a function of the parameter
$h$ of the finite element method.

\section{Acknowledgements}
This paper has been partially sponsored by Project PIP \# 0534 from CONICET-UA, Rosario, Argentina, and AFOSR-SOARD Grant FA9550-14-1-0122.


\end{document}